\documentclass[numbook,envcountsame,draft]{svjour3}
\pdfoutput=1
\usepackage{enumitem}
 \makeatletter
 \if@twocolumn
   \renewcommand\normalsize{%
    \@setfontsize\normalsize\@xpt{12.5pt}%
    \abovedisplayskip=3 mm plus6pt minus 4pt
    \belowdisplayskip=3 mm plus6pt minus 4pt
    \abovedisplayshortskip=0.0 mm plus6pt
    \belowdisplayshortskip=2 mm plus4pt minus 4pt
    \let\@listi\@listI}%

   \renewcommand\small{%
    \@setfontsize\small{8.5pt}\@xpt
    \abovedisplayskip 8.5\p@ \@plus3\p@ \@minus4\p@
    \abovedisplayshortskip \z@ \@plus2\p@
    \belowdisplayshortskip 4\p@ \@plus2\p@ \@minus2\p@
    \def\@listi{\leftmargin\leftmargini
                \parsep 0\p@ \@plus1\p@ \@minus\p@
                \topsep 4\p@ \@plus2\p@ \@minus4\p@
                \itemsep0\p@}%
    \belowdisplayskip \abovedisplayskip}

 \else
   \if@smallext
    \renewcommand\normalsize{%
    \@setfontsize\normalsize\@xpt\@xiipt
    \abovedisplayskip=3 mm plus6pt minus 4pt
    \belowdisplayskip=3 mm plus6pt minus 4pt
    \abovedisplayshortskip=0.0 mm plus6pt
    \belowdisplayshortskip=2 mm plus4pt minus 4pt
    \let\@listi\@listI}%

   \renewcommand\small{%
    \@setfontsize\small\@viiipt{9.5pt}%
    \abovedisplayskip 8.5\p@ \@plus3\p@ \@minus4\p@
    \abovedisplayshortskip \z@ \@plus2\p@
    \belowdisplayshortskip 4\p@ \@plus2\p@ \@minus2\p@
    \def\@listi{\leftmargin\leftmargini
                \parsep 0\p@ \@plus1\p@ \@minus\p@
                \topsep 4\p@ \@plus2\p@ \@minus4\p@
                \itemsep0\p@}%
    \belowdisplayskip \abovedisplayskip}
  \else
   \renewcommand\normalsize{%
    \@setfontsize\normalsize{9.5pt}{11.5pt}%
    \abovedisplayskip=3 mm plus6pt minus 4pt
    \belowdisplayskip=3 mm plus6pt minus 4pt
    \abovedisplayshortskip=0.0 mm plus6pt
    \belowdisplayshortskip=2 mm plus4pt minus 4pt
    \let\@listi\@listI}%

   \renewcommand\small{%
    \@setfontsize\small\@viiipt{9.25pt}%
    \abovedisplayskip 8.5\p@ \@plus3\p@ \@minus4\p@
    \abovedisplayshortskip \z@ \@plus2\p@
    \belowdisplayshortskip 4\p@ \@plus2\p@ \@minus2\p@
    \def\@listi{\leftmargin\leftmargini
                \parsep 0\p@ \@plus1\p@ \@minus\p@
                \topsep 4\p@ \@plus2\p@ \@minus4\p@
                \itemsep0\p@}%
    \belowdisplayskip \abovedisplayskip}
   \fi
 \fi
 
 \makeatother
\smartqed
\usepackage[final]{microtype}
\usepackage{mathptmx}
\DeclareMathAlphabet{\mathcal}{OMS}{cmsy}{m}{n}
\usepackage{graphicx}
\usepackage{amsmath}
\usepackage{mathrsfs}
\usepackage{amssymb}
\usepackage{tikz}
\usepackage{ytableau}
\newcommand{\lside}{_{\mathsf{L}}}

\newcommand{\lessthan}[2]{#1\,{\downarrow}\,#2}
\newcommand{\morethan}[2]{#1\,{\uparrow}\,#2}
\newcommand{\Case}[1]{\smallskip\noindent\textit{Case #1}\endgraf
   \nobreak\smallskip\noindent\ignorespaces}

\newcommand{\tab}{\tau}
\DeclareMathOperator{\STD}{\mathrm{Std}}
\DeclareMathOperator{\Tab}{\mathrm{Tab}}
\DeclareMathOperator{\Skew}{\mathrm{Skew}}
\newcommand{\perm}{\mathrm{perm}}
\newcommand{\squash}{\mathrm{sqsh}}
\newcommand{\moveup}{m}
\newcommand{\row}{\mathrm{row}}
\newcommand{\col}{\mathrm{col}}
\newcommand{\Sym}{\mathrm{Sym}}

\journalname{Algebraic Combinatorics}

\begin{document}
\title{\textit{W}\!-graph determining elements in type \textit{A}}\label{Pre}
\author{Van Minh Nguyen}
\institute{   Van Minh Nguyen \at
              University of Sydney \\
              Fax: +61 2 9351 4534\\
              \email{VanNguyen@.maths.usyd.edu.au}
}
\date{}
\let\makeheadbox\relax
\maketitle

\begin{abstract}
Let \((W,S)\) be a Coxeter system of type \(A\), so that \(W\)
can be identified with the symmetric group \(\Sym(n)\) for
some positive integer~\(n\) and \(S\) with the set of
simple transpositions \(\{\,(i,i+1)\mid 1\leqslant i\leqslant n-1\,\}\). Let
\(\leqslant\lside\) denote the left weak order on \(W\), and for each
\(J\subseteq S\) let \(w_J\) be the longest element of the subgroup
\(W_J\) generated by~\(J\). We show that the basic skew diagrams with \(n\)
boxes are in bijective correspondence with the pairs \((w,J)\) such that the set
\(\{\,x\in W\mid w_J\leqslant\lside x\leqslant\lside ww_J\,\}\) is a nonempty
union of Kazhdan--Lusztig left cells. These are also the pairs \((w,J)\) such
that \(\mathscr{I}(w)=\{\,v\in W\mid v\leqslant\lside w\,\}\)
is a \(W\!\)-graph ideal with respect to~\(J\). Moreover, for each such
pair the elements of \(\mathscr{I}(w)\) are in bijective correspondence
with the standard tableaux associated with the corresponding skew diagram.


\end{abstract}

\section{Introduction}
\label{intro}
Let \((W,S)\) be a Coxeter system and denote by \(\leqslant\lside\) the
left weak order on \(W\) (defined by \(y\leqslant\lside x\) if and only if
\(l(xy^{-1})=l(x)-l(y)\), where \(l\) denotes length relative to~\(S\)).
In~\cite{howvan:wgraphDetSets}
an algorithm was given that takes as input a pair \((\mathscr{I},J)\),
where \(\mathscr{I}\) is an ideal
of \((W,\leqslant\lside)\) and \(J\)
is a subset of \(S\), and produces a graph with edges labelled by integers and
vertices coloured with subsets of~\(S\). If \((\mathscr{I},J)\) is a
\(W\!\)-graph ideal then the output is a \(W\!\)-graph. It was shown
in~\cite{howvan:wgraphDetSets} that \(W\!\)-graphs for the Specht modules
can be produced in this way. In~\cite{nguyen:wgideals2} it was shown,
more generally, that \(W\!\)-graphs for the Kazhdan--Lusztig left cells that
contain longest elements of standard parabolic subgroups can
be obtained from \(W\!\)-graph ideals. Indeed, if \(J \subseteq S\)
and \(W_J\) is the subgroup generated by~\(J\), then the left cell
containing \(w_J\) (the longest element of \(W_J\)) is equal to
\(\mathscr{I}w_J\), where \((\mathscr{I},J)\) is a \(W\!\)-graph ideal. 
 
Our first main result says that if \((W,S)\) is of type~\(A\) and
\(\mathscr I\) is an ideal of \((W,\leqslant\lside)\) then
\(\mathscr{I}w_J\) is a union of Kazhdan--Lusztig left cells whenever
\((\mathscr I,J)\) is a \(W\!\)-graph ideal. Note that, by
\cite[Theorem 9.5]{nguyen:wgbiideals}, this is not
true for types other than~\(A\), even when \((W,S)\) has rank~2. However,
\cite[Theorem 5.2]{nguyen:wgbiideals} shows that, for all types,
if \(\mathcal C\) is a set of left cells that is upward closed, in the sense
that \(c\in\mathcal C\) and \(c'\geqslant c\) implies \(c'\in \mathcal C\),
then \(\bigcup_{c\in\mathcal C}c\) is a \(W\!\)-graph ideal.

Our second main result is the classification, when \((W,S)\) is of type \(A_{n-1}\),
of the pairs \((w,J)\) such that \((\mathscr{I}(w),J)\) is a \(W\!\)-graph ideal,
where \(\mathscr{I}(w) = \{v\in W \mid v\leqslant\lside w\}\). These are exactly
the pairs \((w,J)\) such that \(l(ws)>l(w)\) for all~\(s\in J\) and
\(\mathscr{I}(w)w_J\) is a union of Kahzdan--Lusztig left cells. Furthermore, they
are parametrized by the skew partitions of \(n\), and in each case the elements of
\(\mathscr{I}(w)\) are parametrized by the standard tableaux associated with the
corresponding basic skew diagram.

Since the current work is a sequel to~\cite{nguyen:wgbiideals}, we shall
freely use the notation and terminology of that paper.

\section{Relationship between \textit{W}-graph ideals and Kazhdan--Lusztig
left cells in type \textit{A}}
\label{sec:3}

A complete classification of \(W\!\)-graph ideals of finite Coxeter groups of
rank \(2\) is given in Theorem~9.5 of \cite{nguyen:wgbiideals}. We shall
make use of the following special case.
\begin{lemma}\label{idealA2}
Let \((W,S)\) be a Coxeter system of type \(A_2 = I_2(2)\), and let \(S=\{s,t\}\).
Then \((\mathscr{I}\!,\,J)\) is a \(W\!\)-graph ideal if and only if
one of the following alternatives is satisfied:\setitemindent{xx(viii)}
\begin{itemize}[topsep=1 pt]
\item[\textup{(i)}]\((\mathscr{I}\!,\,J)=(\{1\},S)\),
\item[\textup{(ii)}]\((\mathscr{I}\!,\,J)=(\{1\},\emptyset)\),
\item[\textup{(iii)}]\((\mathscr{I}\!,\,J)=(\{1,t,st\},\{s\})\),
\item[\textup{(iv)}]\((\mathscr{I}\!,\,J)=(\{\{1,t\},\{s\})\),
\item[\textup{(v)}]\((\mathscr{I}\!,\,J)=(\{1,s,ts\},\{t\})\),
\item[\textup{(vi)}]\((\mathscr{I}\!,\,J)=(\{\{1,s\},\{t\})\),
\item[\textup{(vii)}]\((\mathscr{I}\!,\,J)=(\{1,s,t,ts,st,tst\},\emptyset)\),
\item[\textup{(viii)}]\((\mathscr{I}\!,\,J)=(\{1,s,t,ts,st\},\emptyset)\).
\end{itemize}
\end{lemma}

\begin{remark}\label{idealA2rem}
Let \((\mathscr{I}\!,\,J)\) be one of the \(W\!\)-graph ideals in the above list.
It is readily checked that the set \(\mathscr{I}w_J\) contains the element \(t\) if
and only if it also contains the element~\(st\). Similarly, \(s\in\mathscr{I}w_J\)
if and only if \(ts\in\mathscr{I}w_J\). This amounts to saying that 
\(\mathscr{I}w_J\) is a union of left cells of~\(W\), since \(\{t,st\}\) and
\(\{s,ts\}\) are left cells of~\(W\), and the other left cells are singleton sets.
\end{remark}
The following result was proved in \cite[Theorem 8.4]{nguyen:wgbiideals}.

\begin{theorem}\label{restrictedWGideal}
Let \((\mathscr I,J)\) be a \(W\!\)-graph ideal.
Suppose that \(K \subseteq S\) and \(d \in D_{K}^{-1} \cap \mathscr I\!\), and
put \(\mathscr{I}_d = \{\,y \in W_K \mid yd \in \mathscr I\,\}\).
Then \((\mathscr{I}_d,K\cap dJd^{-1})\) is a \(W_{K}\)-graph ideal.
\end{theorem}
Using the notation of \cite{nguyen:wgbiideals}, if \((\mathscr{I}\!,\,J)\) is a
\(W\!\)-graph ideal let \(\Gamma(\mathscr{I}\!,\,J)\) denote the corresponding 
\(W\!\)-graph. We shall identify \(\mathscr I\) with the vertex set
of~\(\Gamma(\mathscr{I}\!,\,J)\). Continuing with the hypotheses of
Theorem~\ref{restrictedWGideal}, let \(\Gamma=\Gamma(\mathscr I,J)\) and
let \(\Gamma_K\) be the \(W_K\)-graph obtained from~\(\Gamma\) by ignoring the
elements of~\(S\setminus K\). By \cite[Remark~8.6]{nguyen:wgbiideals}, the mapping
\(y\mapsto yd\) from
\(\mathscr I_d\) to \(W_Kd\cap\mathscr I\) induces an isomorphism from the
\(W_K\)-graph \(\Gamma(\mathscr I_d,K\cap dJd^{-1})\) to the full subgraph of
\(\Gamma_K\) spanned by \(W_Kd\cap\mathscr{I}\). Moreover, this set is a union
of cells of~\(\Gamma_K\). Since it is trivial that each cell \(X\) of \(\Gamma\)
is a union of cells of~\(\Gamma_K\), it follows that \(Xd^{-1}\cap W_K\)
is a union of cells of \(\Gamma(\mathscr I_d,K\cap dJd^{-1})\). Applying this
in the case \(\mathscr I=D_J\) gives the following result.

\begin{proposition}\label{cellrestrict}
Let \(J,\,K\subseteq S\) and \(d\in D_{K,J}\), and let \(M=K\cap dJd^{-1}\).
If \(X\) is any cell of~\((D_J,J)\), then \(Xd^{-1}\cap W_K\) is a union of cells
of~\((D_M^K,M)\), where \(D_M^K=W_K\cap D_M\).
\end{proposition}
\begin{remark} Proposition~\ref{cellrestrict} also follows from 
\cite[Proposition 5.7]{howyin:indwg2}, taking \(\Gamma\) to be the
single vertex \(W_J\)-graph corresponding to the trivial representation
of~\(W_J\).
\end{remark}
Assume now that \((W,S)\) is of type~\(A\). Following the terminology of
\cite[Section 5]{kazlus:coxhecke}, for \(x\in W\) define
\(\mathscr{L}(x)=\{\,s\in S\mid sx<x\}\), and let \(\approx\) be the equivalence
relation on \(W\) generated by the relations \(x\approx sx\) for all \(x\in W\) and 
\(s\in S\) such that \(x<sx\) and \(\mathscr{L}(x)\nsubseteq\mathscr{L}(sx)\).
Kazhdan and Lusztig show that \(\approx\) coincides with Kazhdan--Lusztig left equivalence,
so that the equivalence classes are precisely the left cells. Hence to show that
a subset \(X\) of \(W\) is a union of left cells it is sufficient to show that
whenever the relation \(x\approx sx\) holds, \(x\in X\) if and only if \(sx\in X\). 

\begin{lemma}\label{celllemma}
Let \((W,S)\) be a Coxeter system of type~\(A\) and let \(X\subseteq W\). Then
\(X\) is a union of left cells if and only if for all \(s,\,t\in S\) such that
\(st\) has order~\(3\) and all \(d\in W\) such that \(sd>d\) and \(td>d\), the
set \(\{\,y\in W_{\{s,t\}}\mid yd\in X\,\}\) is a union of
left cells in~\(W_{\{s,t\}}\).
\end{lemma}

\begin{proof}
Suppose first that \(X\) is a union of left cells. Let \(s,\,t\in S\) with \(st\) of
order~\(3\), and let \(d\in W\) satisfy \(sd>d\) and \(td>d\). We must show that
the set \(X_d=\{\,y\in W_{\{s,t\}}\mid yd\in X\,\}\) is a union of left cells
of~\(W_{\{s,t\}}\). That is, we must show that \(t\in X_d\) if and only if
\(st\in  X_d\), and \(s\in  X_d\) if and only if \(ts\in  X_d\). But since 
\(l(xd)=l(x)+l(d)\) for all \(x\in W_{\{s,t\}}\) it is clear that \(td\approx std\),
since \(l(td)<l(std)\) and \(t\in\mathscr{L}(td)\setminus\mathscr{L}(std)\). So
\(td\) and \(std\) are in the same left cell of~\(W\), and hence
\(td\in X\) if and only if \(std\in X\). Similarly, \(sd\in X\) if and only if
\(tsd\in X\). Thus \(t\in X_d\) if and only if \(st\in X_d\), and
\(s\in X_d\) if and only if \(ts\in X_d\), giving the desired conclusion.

Conversely, assume that for all \(s,\,t\in S\) with \(st\) of order~\(3\) and all
\(d\in W\) with \(sd>d\) and \(td>d\), the set
\(\{\,y\in W_{\{s,t\}}\mid yd\in X\,\}\) is a union of left cells in~\(W_{\{s,t\}}\).
We must show that \(X\) is a union of \(\approx\) equivalence classes. 
It suffices to show that if \(x\in W\) and \(s\in S\)
satisfy \(x<sx\) and \(\mathscr{L}(x)\nsubseteq\mathscr{L}(sx)\), then \(x\in X\)
if and only if \(sx\in X\). Given such elements \(x\) and~\(s\), choose \(t\in S\)
such that \(t\in\mathscr{L}(x)\setminus\mathscr{L}(sx)\). Then \(tx<x<sx<tsx\), so
that \(l((tst)(tx))=l(tsx)=l(tx)+3\). So \(st\) has order~\(3\), and
\(d=tx\) is the shortest element of the coset \(W_{\{s,t\}}d\) (and
\(tsx=tstd\) the longest). Since
\(X_d=\{\,y\in W_{\{s,t\}}\mid yd\in X\,\}\) is a union of left cells
of~\(W_{\{s,t\}}\), it follows that \(t\in X_d\) if and only if~\(st\in X_d\). That is,
\(td\in X\) if and only \(std\in X\), as required.
\qed
\end{proof}

\begin{theorem}\label{generalisedInd}
Let \((W,S)\) be a Coxeter system of type~\(A\), and let
\((\mathscr{I},J)\) be a \(W\!\)-graph ideal. Then \(\mathscr{I}w_{J}\) is a
union of Kazhdan--Lusztig left cells.
\end{theorem}

\begin{proof} We use Lemma~\ref{celllemma}. Accordingly,
let \(s,\,t\in S\) be such that \(st\) has order~\(3\) and let \(d\in W\) be
such that \(sd>d\) and \(td>d\). Put \(K=\{s,t\}\) and
\(Y=\{\,y\in W_K\mid yd\in\mathscr{I}w_J\,\}\). We show that \(Y\) is a union of
left cells in~\(W_K\), noting first that this certainly holds if \(Y=\emptyset\).

Assume that \(Y\ne\emptyset\). Let \(e\) be the minimal length element in the 
right coset \(W_Kdw_J\). Since \(\mathscr{I}\cap W_Kdw_J = Ydw_J\)
is nonempty and \(\mathscr{I}\) is an ideal of \((W,\leqslant\lside)\) it follows that
\(e\in\mathscr{I}\), and hence
\(e\in D_K^{-1}\cap\mathscr{I}\subseteq D_K^{-1}\cap D_J\). By
Theorem~\ref{restrictedWGideal} it follows that \((\mathscr{I}_e,L)\) is
a \(W_K\)-graph ideal, where \(\mathscr{I}_e=\{\,y\in W_K\mid ye\in\mathscr{I}\,\}\)
and \(L=K\cap eJe^{-1}\). Hence \(\mathscr{I}_ew_L\) is a union of left cells
in \(W_K\) by Remark~\ref{idealA2rem}.

It remains to observe that \(\mathscr{I}_ew_L=Y\), and since
\(\mathscr{I}_ew_L=\{\,y\in W_K\mid yw_Le\in\mathscr{I}\,\}\) it suffices
to show that \(d=w_Lew_J\). Since \(e\) is the minimal length element in
\(W_KeW_J\) and \(L=K\cap eJe^{-1}\) it follows that \(w_Lew_J\) is the
minimal length element in \(W_Kew_J=W_Kd\), and hence \(d=w_Lew_J\),
as required.\qed
\end{proof}

\begin{remark}\label{D_Jcells}
It was shown in \cite[Proposition 5.13]{nguyen:wgbiideals} that
if \(J\) is any subset of~\(S\) then \(\Gamma(D_J,J)\) is isomorphic to the
full subgraph of \(\Gamma(W,\emptyset)\) spanned by the vertices corresponding
to elements of~\(D_Jw_J\), via the obvious bijection~\(D_J\to D_Jw_J\).
Moreover, \(D_Jw_J\) is a union of left cells of~\(W\),
and \(X\subseteq D_Jw_J\) is a left cell of~\(W\) if and only if
\(\{\,w\in D_J\mid ww_J\in X\,\}\) is a cell of \((D_J,J)\). Thus 
Theorem~\ref{generalisedInd} tells us that if \((W,S)\) is of type~\(A\)
and \((\mathscr{I}\!,\,J)\) is a \(W\!\)-graph ideal then \(\mathscr{I}\) is
a union of cells of \((D_J,J)\).
Now by \cite[Theorem 5.2]{nguyen:wgbiideals}, if \(\cal{C}\) is a set of cells
of \((D_J,J)\) that is upward closed with respect to the Kazhdan--Lusztig
partial order on cells, then the union of the cells in \(\cal{C}\) is
an ideal of \((W,\leqslant\lside)\) and is a \(W\!\)-graph ideal with respect
to~\(J\). (These are the strong \(W\!\)-graph subideals of \((D_J,J)\).)
We conjecture that in type~\(A\) all \(W\!\)-graph ideals have this form.

It has been shown by calculation that when \(W\) is of type \(A_5\) there exist
ideals of \((W,\leqslant\lside)\) that are unions of left cells but are not
\(W\!\)-graph ideals. It is well known that in this case the Robinson-Schensted
map \(w\mapsto (P(w),Q(w))\) is a bijection from \(W\) to the set of ordered
pairs of standard tableaux on \(\{1,2,3,4,5,6\}\), and that for each such standard
tableau~\(t\) the set \(\{\,w\in W\mid Q(w)=t\,\}\) is a left cell of~\(W\).
(See \cite[Theorem A]{ariki:RSleftcells}.) Using the computational algebra system
Magma (see~\cite{Magma}), R.~B.~Howlett has shown that if \(\mathscr I\) is the union
of the left cells corresponding to the tableaux listed below then \(\mathscr I\)
is an ideal of~\((W,\leqslant\lside)\), but \((\mathscr I\!,\,\emptyset)\) is not
a \(W\!\)-graph ideal.
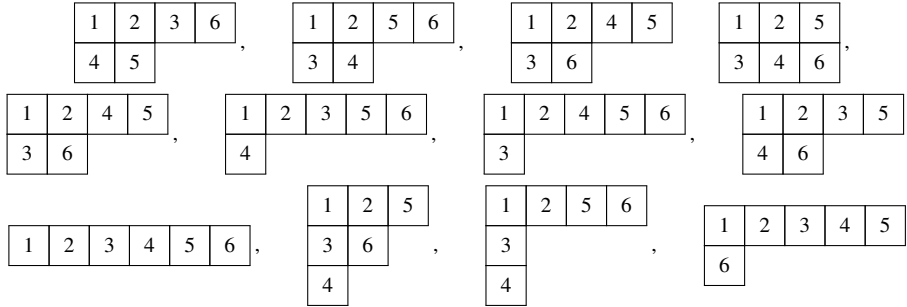
\ytableausetup{aligntableaux=center}
\begin{gather*}\advance\belowdisplayskip-2 pt
\text{\small \begin{ytableau}
1&2&3&6\\4&5
\end{ytableau}\ ,
\qquad
\begin{ytableau}
1&2&5&6\\3&4
\end{ytableau}\ ,
\qquad
\begin{ytableau}
1&2&4&5\\3&6
\end{ytableau}
\qquad
\begin{ytableau}
1&2&5\\3&4&6
\end{ytableau}\ ,}
\\
\text{\small \begin{ytableau}
1&2&4&5\\3&6
\end{ytableau}\ ,
\qquad
\begin{ytableau}
1&2&3&5&6\\4
\end{ytableau}\ ,
\qquad
\begin{ytableau}
1&2&4&5&6\\3
\end{ytableau}\ ,
\qquad
\begin{ytableau}
1&2&3&5\\4&6
\end{ytableau}\, ,}
\\
\text{\small \begin{ytableau}
1&2&3&4&5&6
\end{ytableau}\ ,
\qquad
\begin{ytableau}
1&2&5\\3&6\\4
\end{ytableau}\ ,
\qquad
\begin{ytableau}
1&2&5&6\\3\\4
\end{ytableau}\ ,
\qquad
\begin{ytableau}
1&2&3&4&5\\6
\end{ytableau}\ .}
\end{gather*}
\end{remark}

\section{\textit{W-}graphs derived from skew partitions}
\label{sec:4}
For each positive integer~\(n\) we define \(W_n\) to be the symmetric
group on the set \(\{1,2,\ldots,n\}\) and \(S_n=\{\,s_i\mid 1\leqslant i<n\,\}\),
where \(s_i\) is the transposition~\((i,i+1)\). Thus \((W_n,S_n)\) is
a Coxeter system of type~\(A_{n-1}\). If \(l\) and \(m\)
are positive integers then we define \(W_{[l,m]}\) to be the set of
all permutations of \([l,m]=\{\,i\in\mathbb{Z}\mid l\leqslant i\leqslant m\,\}\).
If \([l,m]\subseteq[1,n]\) then \(W_{[l,m]}\) can be regarded as a standard
parabolic subgroup of~\(W_n\) (generated by \(\{\,s_i\in S_n\mid l\leqslant i<m\,\}\)). We
use a left operator convention for permutations,
writing \(wi\) for the image of~\(i\) under the permutation~\(w\).

A \textit{partition\/}~of \(n\) is a sequence of positive integers
\(\lambda_{1}\geqslant\lambda_{2}\geqslant\cdots\geqslant\lambda_{k}\) with \(\sum_{i=1}^k\lambda_i= n\).
The \(\lambda_i\) are called the \textit{parts\/} of the partition. We adopt the
convention that if \(\lambda\) is a partition with \(k\) parts then \(\lambda_i\)
denotes the \(i\)-th part of~\(\lambda\) if \(i\in\{1,2,\ldots k\}\), and
\(\lambda_i=0\) if \(i>k\). The \textit{Young diagram\/} of \(\lambda\) is
the set
\[
\advance\abovedisplayskip-2 pt\advance\belowdisplayskip-2 pt
[\lambda]=\{\,(i,j)\mid1\leqslant i\leqslant k\text{ and }1\leqslant j\leqslant\lambda_{i}\,\},
\]
represented pictorially as a left-justified array of boxes with
\(\lambda_i\) boxes in the \(i\)-th row from the top. We define
\(P(n)\) to be the set of all partitions of~\(n\).

If \(\lambda\in P(n)\) then \(\lambda^*\) denotes the \textit{conjugate\/} of \(\lambda\),
defined to be the partition whose diagram is the transpose of~\([\lambda]\).
That is, \([\lambda^*]=\{\,(j,i)\mid (i,j)\in [\lambda]\,\}\).

A \textit{skew partition of \(n\)} is an ordered pair \(\lambda/\mu\) such that
\(\lambda\in P(m+n)\) and~\(\mu\in P(m)\) for some nonnegative integer~\(m\), and
\(\lambda_i\geqslant\mu_i\) for all~\(i\). We write \(\lambda/\mu\vdash n\) to indicate
that \(\lambda/\mu\) is a skew partition of \(n\). In the case \(m=0\) we identify
\(\lambda/\mu\) with \(\lambda\).

The \textit{skew diagram\/} \([\lambda/\mu]\) corresponding  to a skew partition
\(\lambda/\mu\) is defined to be the complement of
\([\mu]\) in \([\lambda]\). That is,
\[
\advance\abovedisplayskip-2 pt\advance\belowdisplayskip-2 pt\postdisplaypenalty=10000
[\lambda/\mu]=\{\,(i,j)\mid (i,j)\in[\lambda]\text{ and }(i,j)\notin[\mu]\,\}.
\]
We say that \([\lambda/\mu]\) has \(\lambda_1\) columns and \(\lambda_1^*\) rows,
and that \([\lambda/\mu]\) is \textit{basic\/} if all rows and columns are nonempty.
Thus \([\lambda/\mu]\) is basic if \(\lambda_i>\mu_i\) and \(\lambda_j^*>\mu_j^*\)
for all \(i\leqslant \lambda_1^*\) and \(j\leqslant\lambda_1\).

If \(\lambda/\mu\vdash n\) then a \textit{skew tableau of shape \(\lambda/\mu\)},
or \((\lambda/\mu)\)-tableau, is a bijective map \(t\colon[\lambda/\mu]
\rightarrow\mathcal{A}\), where \(\mathcal{A}\) is a totally ordered
set with~\(n\) elements. We call \(\mathcal{A}\) the \textit{target\/} of~\(t\).
In this paper the target will always be an interval \([m+1,m+n]\), with \(m=0\)
unless otherwise specified. For each \(a\in\mathcal{A}\) we define \(\row(t,a)\)
and \(\col(t,a)\) to be the row index and column index of \(a\)~in~\(t\), so that
\(t^{-1}(a)=(\row(t,a),\col(t,a))\).
We say that \(t\) is \textit{row standard\/}
if its entries increase across the rows, \textit{column standard\/} if
its entries increase down the columns, and \textit{standard\/} if
it is both row standard and column standard.

We define
\(\Tab_m(\lambda/\mu)\) to be the set of all \((\lambda/\mu)\)-tableaux
with target~\([m+1,m+n]\), and
\(\STD_m(\lambda/\mu)=\{\,t\in\Tab_m(\lambda/\mu)\mid t\text{ is standard}\,\}\).
The subscript \(m\) is usually
omitted if \(m=0\). If \(h\in\mathbb{Z}\) and \(t\in\Tab_m(\lambda/\mu)\) then 
we define \(h+t\in\Tab_{h+m}(\lambda/\mu)\) to be the tableau obtained by adding
\(h\) to all entries of~\(t\). We define \(\tab^{\lambda/\mu}\in \STD(\lambda/\mu)\)
to be the \((\lambda/\mu)\)-tableau given by
\begin{equation}\label{toptab}
\tab^{\lambda/\mu}(i,j)=j-\mu_i+\sum_{h=1}^{i-1}(\lambda_h-\mu_h)
\end{equation}
for all \((i,j)\in[\lambda/\mu]\).
That is, the numbers
\(1,\,2,\,\dots,\,(\lambda_1-\mu_1)\) fill the first row of
\(\tab^{\lambda/\mu}\) in order from left to right, then the numbers
\((\lambda_1-\mu_1)+1,\,(\lambda_1-\mu_1)+2,\,\dots,\,(\lambda_1-\mu_1)+(\lambda_2-\mu_2)\)
similarly fill the second row, and so on. We also define \(\tab_{\lambda/\mu}\)
to be the standard \((\lambda/\mu)\)-tableau that is the transpose of the
\((\lambda^*/\mu^*)\)-tableau \(\tab^{\smash{\lambda^*/\mu^*}}\)
(so that the numbers \(1\) to \(n\) fill the columns of
\(\tab_{\lambda/\mu}\) in order from left to right).

It is clear that if \(\lambda/\mu\vdash n\) then
\(W_{[m+1,m+n]}\) acts on \(\Tab_m(\lambda/\mu)\), via
\((wt)(i,j) = w(t(i,j))\) for all \((i,j)\in[\lambda/\mu]\), all
\(t\in\Tab_m(\lambda/\mu)\) and all \(w \in W_{[m+1,m+n]}\). Thus \(w\mapsto
w(m+\tab_{\lambda/\mu})\) gives a bijective map from \(W_{[m+1,m+n]}\) to
\(\Tab_m(\lambda/\mu)\). We define \(\perm{:}\,\Tab_m(\lambda/\mu)\to W_{[m+1,m+n]}\)
to be the map inverse to \(w\mapsto w(m+\tab_{\lambda/\mu})\), and use \(\perm\)
to transfer the Bruhat order and the left weak order from \(W_{[m+1,m+n]}\) to
\(\Tab_m(\lambda/\mu)\). That is,
for all \(t_1,\,t_2\in\Tab_m(\lambda/\mu)\) we write
\(t_1\leqslant t_2\) if and only if \(\perm(t_1)\leqslant \perm(t_2)\), and
\(t_1\leqslant\lside t_2\) if and only if \(\perm(t_1)\leqslant\lside\perm(t_2)\).

\begin{remark}
If \(\lambda/\mu\vdash n\) and \(t\in\STD(\lambda/\mu)\) then the 
\textit{reading word\/} of \(t\) is defined to be the sequence
\((a_1,a_2,\ldots,a_n)\) obtained by concatenating the rows of \(t\)
in order from last row to first row. So there is a bijection \(\STD(\lambda/\mu)\to W_n\)
that maps each \(t\) to the permutation \(\mathrm{word}(t)\) defined by
\(i\mapsto a_i\) for all \(i\in\{1,2,\ldots,n\}\). This bijection and our bijection
\(\perm\) are related by the equation
\(\perm(t)=\mathrm{word}(t)w^{-1}\), where \(w=\mathrm{word}(\tab_{\lambda/\mu})\).
\end{remark}
Whenever \(\lambda/\mu\vdash n\) we define
\(J_{\lambda/\mu}\) to be the subset of \(S_n\) consisting of those
\(s_i\) such that \(i\) and \(i+1\) lie in the same
column of~\(\tab_{\lambda/\mu}\), and we define \(W_{\lambda/\mu}\) to be the
standard parabolic subgroup of \(W_n\) generated by~\(J_{\lambda/\mu}\).
Thus \(W_{\lambda/\mu}\) is the column group of~\(\tab_{\lambda/\mu}\).
Moreover, the set
\(D=\{\,d\in W_n\mid di<d(i+1)\text{ whenever \(s_i\in
J_{\lambda/\mu}\)}\,\}\)
is the set of minimal length representatives of the left cosets of
\(W_{\lambda/\mu}\) in \(W_n\), since the condition \(di<d(i+1)\) is equivalent to
\(l(ds_i)>l(d)\).
Thus \(\{\,d\tab_{\lambda/\mu}\mid d\in D\,\}\) is precisely the set
of column standard \((\lambda/\mu)\)-tableaux.

We shall make use of the following result, which was was proved
in \cite{nguyen:wgideals2}.

\begin{theorem}\textup{\cite[Theorem 6.6]{nguyen:wgideals2}}\label{wdetelmA}
Let \(\lambda\in\mathcal P(n)\) and \(J=J_\lambda\), so that \(W_J=W_\lambda\)
is the column group of~\(\tab_\lambda\). Then
\(y\in W_n\) given by \(y\tab_\lambda=\tab^\lambda\) is a \(W_n\)-graph
determining element (with respect to \(J\)), and its \(W_n\)-graph is isomorphic to the
\(W_n\)-graph of the left cell that contains~\(w_{J}\).
\end{theorem}

\begin{remark}\label{strongspecht}
In our present notation, Theorem~\ref{wdetelmA} says that
\((\mathscr{I}\!,\,J)=(\mathscr{I}(\perm(\tab^\lambda)),J_\lambda)\) is a
\(W_n\)-graph ideal, and that \(\mathscr{I}\) is isomorphic to the left
cell of~\(W_n\) containing~\(w_J\). As explained in Remark~\ref{D_Jcells} above,
the left cell of~\(W_n\) containing~\(w_J\) is isomorphic to the cell of
\((D_J,J)\) containing~\(1\). By \cite[Proposition~6.5]{nguyen:wgideals2},
this cell coincides with~\(\mathscr{I}\). Now if \(X\) is an arbitrary
cell of~\((D_J,J)\) then \(X\leqslant\mathscr{I}\), since
\(1\leqslant\lside d\) for all~\(d\in D_J\). Thus \(\{\mathscr{I}\}\) is an
upward-closed set of cells. So in fact the results of
\cite{nguyen:wgideals2} show that \((\mathscr{I}\!,\,J)\) is a strong
\(W\!\)-graph subideal of~\((D_J,J)\).
\end{remark}
The following result is a straightforward generalization of~\cite[Lemma 6.2]{howvan:wgraphDetSets}.
\begin{lemma}\label{stdideal}
Let \(\lambda/\mu\vdash n\) and let \(\mathscr{I}=\mathscr{I}(\perm(\tab^{\lambda/\mu}))\),
the ideal of \((W_n,\leqslant\lside)\) generated by \(\perm(\tab^{\lambda/\mu})\).
Then \(\STD(\lambda/\mu)=\{\,w\tab_{\lambda/\mu} \mid w\in\mathscr{I}\,\}
=\{\,t \in \Tab(\lambda/\mu) \mid t\leqslant\lside \tab^{\lambda/\mu}\,\}\).
\end{lemma}

\begin{remark}
The fact that the set of standard \((\lambda/\mu)\)-tableaux is in one-to-one
correspondence with the ideal of \((W_{n},\leqslant\lside)\) generated by the
element \(\perm(\tab^{\lambda/\mu})\) is included, with other results,
in Theorems 7.2 and 7.5 of~\cite{BjoWachs:quotientCox}.
\end{remark}
\begin{definition}
We call \(m+\tab^{\lambda/\mu}\)
the \textit{maximal tableau\/} in \(\STD_m(\lambda/\mu)\).
\end{definition}
Let \(\lambda /\mu\vdash n\) and \(t\in\STD_m(\lambda/\mu)\).
For each \(k\in [m+1,m+n]\) we define \(\lessthan tk\) to be the skew tableau
obtained by removing from \(t\) all boxes filled with entries greater than or
equal to~\(k\). Thus \(\lessthan tk\in\STD_m(\kappa/\mu)\), where
\([\kappa]=[\lambda]\setminus\{\,(i,j)\in[\lambda/\mu] \mid t(i,j)\geqslant k\,\}\), 
and \(\lessthan tk\) is the restriction of \(t\)
to \([\kappa/\mu] = \{\,(i,j)\in[\lambda/\mu] \mid t(i,j)< k\,\}\). Clearly
\(\kappa/\mu\vdash(k-m-1)\). Similarly,
we define \(\morethan tk\) to be the skew tableau obtained by removing
all boxes with entries less than or equal to~\(k\), so that
\(\morethan tk\in\STD_k(\lambda/\nu)\) where
\([\nu]=[\mu]\cup\{\,(i,j)\in[\lambda/\mu] \mid t(i,j)\leqslant k\,\}\), 
and \(\morethan tk\) is the restriction of \(t\)
to \([\lambda/\nu] = \{\,(i,j)\in[\lambda/\mu] \mid t(i,j)> k\,\}\).
Clearly \(\lambda/\nu\vdash(m+n-k)\).

We can now show that Theorem~\ref{wdetelmA} generalizes to skew partitions
in the natural way.

\begin{theorem}\label{finalThr}
Let \(\lambda/\mu\vdash n\) and \(M=J_{\lambda/\mu}\).
Then \((\mathscr{I}(\perm(\tab^{\lambda/\mu})),M)\) is a strong
\(W_n\)-graph subideal of \((D_M,M)\).
\end{theorem}

\begin{proof}
Recall that \(\lambda\vdash(m+n)\) and \(\mu\vdash m\) for some nonnegative
integer~\(m\). Clearly there is an isomorphism
\(\varphi\colon W_n\to W_{[m+1,m+n]}\) given by \(s_i\mapsto s_{m+i}\) for all
\(i\in[1,n-1]\), and it suffices to prove that
\((\varphi(\mathscr{I}(\perm(\tab^{\lambda/\mu}))),\varphi(M))\) is a strong
\(W_{[m+1,m+n]}\)-graph subideal of \((\varphi(D_M),\varphi(M))\).
Note that \(\mathscr{I}(\perm(\tab^{\smash{\lambda/\mu}}))
=\{\,w\in W_n\mid w\tab_{\lambda/\mu}\in\STD(\lambda/\mu)\,\}\),
by Lemma~\ref{stdideal}, and so
\(\varphi(\mathscr{I}(\perm(\tab^{\smash{\lambda/\mu}})))
=\{\,w\in W_{[m+1,m+n]}\mid w(m+\tab_{\lambda/\mu})\in\STD_m(\lambda/\mu)\,\}\).

Write \(W=W_{m+n}\), and note that \(W_{[m+1,m+n]}\) can be identified with the
standard parabolic subgroup \(W_K\) of~\(W\), where \(K = \{\,s_i\mid m+1\leqslant i< m+n\,\}\).
Let \(t_0\in \STD(\lambda)\) be defined by \(\lessthan{t_0}{(m+1)}=\tab_\mu\)
and \(\morethan{t_0}m=m+\tab_{\lambda/\mu}\), and let \(d=\perm(t_0)\). Thus
\(t_0=d\tab_\lambda\) and \(d\in D_J\), where \(J=J_\lambda\). Since
\((d^{-1}(m+1),d^{-1}(m+2),\ldots,d^{-1}(m+n))\) is an increasing sequence,
consecutive terms differing either by~1 or by \(1+\mu_j^*\) for some~\(j\), it
follows that \(sd>d\) for all~\(s\in K\). Thus \(d\in D_K^{-1}\), and so
\(d\in D_{K,J}=D_K^{-1}\cap D_J\). By \cite[Lemma~2.5]{nguyen:wgbiideals},
this implies that \(\{\,w\in W_K\mid wd\in D_J\,\}=D^K_{K\cap dJd^{-1}}\),
the set of minimal length representatives of the left cosets of
\(W_{K\cap dJd^{-1}}\) in~\(W_K\).

Write \(\mathscr{I}=\mathscr{I}(\perm(\tab^{\smash{\lambda}}))\), so that
\((\mathscr{I},J)\) is a \(W\!\)-graph ideal by Theorem~\ref{wdetelmA}, and a
strong \(W\!\)-graph subideal of \((D_J,J)\) by Remark~\ref{strongspecht}.
Furthermore, \(x\mapsto x\tab_\lambda\) gives a bijection from
\(\mathscr{I}\)~to~\(\STD(\lambda)\). Since 
\begin{align*}
\{\,w\in W_K\mid wd\in\mathscr{I}\,\}&=
\{\,w\in W_K\mid wd\tab_{\lambda}\in\STD(\lambda)\,\}\\
&=\{\,w\in W_K\mid wt_0\in\STD(\lambda)\,\}\\
&=\{\,w\in W_K\mid \morethan{(wt_0)}m\in\STD_m(\lambda/\mu)\,\}\\
&=\{\,w\in W_K\mid w(m+\tab_{\lambda/\mu})\in\STD_m(\lambda/\mu)\,\}\\[-1\jot]
&=\varphi(\mathscr{I}(\perm(\tab^{\lambda/\mu}))),
\end{align*}
it follows from Theorem~\ref{restrictedWGideal} that
\((\mathscr{I}(\varphi(\perm(\tab^{\lambda/\mu}))),K\cap dJd^{-1})\)
is a \(W_K\)-graph ideal. Moreover, by \cite[Theorem~8.7]{nguyen:wgbiideals},
it is a strong \(W_K\)-graph subideal of \((D^K_{K\cap dJd^{-1}},K\cap dJd^{-1})\).
To complete the proof, it remains to show that
\(K\cap dJd^{-1}=\varphi(M)\). But this is clear from the fact
that \(\morethan{d\tab_{\lambda}}m=m+\tab_{\lambda/\mu}\), which shows
that if \(m+1\leqslant i<m+n\) then \(i\) and \(i+1\) are
in the same column of \(d\tab_\lambda\) if and only if \(i-m\) and \(i-m+1\)
are in the same column of \(\tab_{\lambda/\mu}\).
\qed
\end{proof}

\section{Converse of Theorem \ref{finalThr}}
\label{sec:9}
Continuing with the notation used in the previous section, for each
\(J\subseteq\{s_1,s_2,\ldots,s_{n-1}\}\) define \(\Skew(J)\) to be the set
of all \(\lambda/\mu \vdash n\) such that \(J_{\lambda/\mu}=J\) and the 
skew diagram \([\lambda/\mu]\) is basic. For each \(w\in D_J\) define
\(\Skew(J,w)=\{\,\lambda/\mu\in\Skew(J)\mid w\tab_{\lambda/\mu}\in\STD(\lambda/\mu)\,\}\),
and \(\STD(J,w)=\{\,w\tab_{\lambda/\mu}\mid \lambda/\mu\in\Skew(J,w)\,\}\).
Observe that if \(t\in\STD(J,w)\) then all other elements of \(\STD(J,w)\) can be
obtained from~\(t\) by sliding columns vertically up or down. Clearly
\(\STD(J,w)\ne\emptyset\), since \(\Skew(J,w)\) always contains the (unique)
\(\lambda/\mu\in\Skew(J)\) such that \([\lambda/\mu]\) has \(n\)
rows, all of length~1. We define \(\tau(J,w)\) to
be the element of \(\STD(J,w)\) with the least possible number of rows. It is clear
that this element is unique.

In this section we shall show that if \(w\in W_n\) has the property that
\((\mathscr{I}(w),J)\) is a \(W_n\)-graph ideal then \(\tau(J,w)=\tab^{\lambda/\mu}\)
for some \(\lambda/\mu\vdash n\). 

For the purposes of this section it is convenient to make the following definition.

\begin{definition} Let \((W,S)\) be a Coxeter system, let \(K\subseteq S\), and let
\(W_K\) be the subgroup of~\(W\) generated by~\(K\). If \(w\in W\) then
the \textit{left \(W_K\)-component of \(w\)} is the element \(x\in W_K\) such
that \(x^{-1}w\) is the minimal length element of~\(W_Kw\), and \(d=x^{-1}w\)
is the \(D_K^{-1}\)-\textit{component of~\(w\)}.
\end{definition}
Let \(W=W_n\) and let \(K=S_n\setminus\{s_{n-1}\}\) and \(L=S_n\setminus\{s_1\}\), so
that the parabolic subgroups \(W_K\) and \(W_L\) of~\(W\) can be identified with
\(W_{n-1}\) and \(W_{[2,n]}\) respectively. Let \(\lambda/\mu\vdash n\) and
\(t\in\STD(\lambda/\mu)\). Since \(t\) is standard, the number~\(n\) must be at the
bottom of its column and the right hand end of its row in~\(t\), and the
number~\(1\) must be at the top of its column and the left hand end of its
row. Define \(\kappa\) and \(\nu\) by \([\kappa]=[\lambda]\setminus\{t^{-1}(n)\}\)
and \([\nu]=\mu\cup\{t^{-1}(1)\}\). The restriction of \(t\) to \([\kappa/\mu]\) is
\(\lessthan tn\), the \((\kappa/\mu)\)-tableau obtained by deleting from~\(t\)
the box containing~\(n\), and the restriction of \(t\) to \([\lambda/\nu]\)
is \(\morethan t1\), the \((\lambda/\nu)\)-tableau obtained by deleting
from~\(t\) the box containing~\(1\). Clearly \(\lessthan tn\in\STD(\kappa/\mu)\)
and \(\morethan t1\in\STD_1(\lambda/\nu)\). 

\begin{lemma}\label{restriction}
Let \(W=W_n\) and let \(K=S_n\setminus\{s_{n-1}\}\) and \(L=S_n\setminus\{s_1\}\).
Let \(\lambda/\mu\) be a skew partition of~\(n\) and \(t\in\STD(\lambda/\mu)\),
and let \(w=\perm(t)\).\setitemindent{(iii)}
\begin{itemize}[topsep=1 pt]
\item[\textup{(i)}]The left \(W_K\)-component of~\(w\) is \(\perm(\lessthan tn)\),
and the left \(W_L\)-component of~\(w\) is \(\perm(\morethan t1)\).
\item[\textup{(ii)}]Let \([\lambda]\setminus\{t^{-1}(n)\}=[\kappa]\)
and \(\mu\cup\{t^{-1}(1)\}=[\nu]\),
and let \(t',\,t''\in\STD(\lambda/\mu)\) be defined by
\(\lessthan{t'}n=\tab_{\smash{\kappa/\mu}}\) and \(t'(n)=t(n)\), and 
\(\morethan{t''}1=1+\tab_{\smash{\lambda/\nu}}\) and \(t''(1)=t(1)\).  Then the
\(D_K^{-1}\)-component 
of~\(w\) is \(d=\perm(t')\) and the \(D_L^{-1}\)-component of \(w\) is \(e=\perm(t'')\). 
Furthermore, \(K\cap dJ_{\lambda/\mu}d^{-1}=J_{\kappa/\mu}\) and
\(L\cap eJ_{\lambda/\mu}e^{-1}=\{\,s_i\in L\mid s_{i-1}\in J_{\lambda/\nu}\,\}\).
\end{itemize}
\end{lemma}

\begin{proof}
Let \(d=\perm(t')\) and let \(t^{-1}(n)=(i,j)\in [\lambda/\mu]\).
Put \(q=\tab_{\lambda/\mu}(i,j)\), and note that \(w^{-1}(n)=d^{-1}(n)=q\) and
\((d^{-1}(1),d^{-1}(2),\ldots,d^{-1}(n-1))\) is the sequence obtained by
deleting \(q\) from the sequence \((1,2,\ldots,n)\). Thus \(wd^{-1}\in W_K\)
and \(l(sd)>l(d)\) for all \(s\in K\), which shows that \(d\) is the
\(D_K^{-1}\)-component of~\(w\).

Since the \(D_K^{-1}\)-component of \(w\) is \(d\), the left \(W_K\)-component of
\(w\) is \(wd^{-1}\). Since \(wd^{-1}(t')=w\tab_{\lambda/\mu}=t\) it
follows that \(wd^{-1}(\tab_{\kappa/\mu})=wd^{-1}(\lessthan{t'}n)\) is the
tableau obtained from \(t\) by deleting the box containing \(wd^{-1}(n)=n\). Hence
\(wd^{-1}=\perm(\lessthan tn)\).

The proof that \(e=\perm(t'')\) is the \(D_L^{-1}\)-component of~\(w\) and the proof
that \(\perm(\morethan t1)\) is the left \(W_L\)-component of~\(w\) are similar to
the proofs that \(d\) is the \(D_K^{-1}\)-component and \(\perm(\lessthan tn)\) the left
\(W_K\)-component. So it remains to prove the last sentence of~(ii). 

Let \(J=J_{\lambda/\mu}\), so that \(W_J=W_{\lambda/\mu}\) is the column
group of the tableau \(\tab_{\lambda/\mu}\). Then \(eW_Je^{-1}\) is the column
group of \(e\tab_{\lambda/\mu}=t''\), and \(W_L\cap eW_Je^{-1}\) is
the column group of \(\morethan {t''}1=1+\tab_{\lambda/\nu}\). That is,
\(W_L\cap eW_Je^{-1}=W_M\), where \(M=\{\,s_i\in L\mid s_{i-1}\in J_{\lambda/\nu}\,\}\).
But \(W_L\cap eW_Je^{-1}=W_{L\cap eJe^{-1}}\),
by \cite[Theorem 2.7.4]{Carter:book}, since \(e\in D_L^{-1}\cap D_J\),
and hence \(L\cap eJe^{-1}=M\), as required.

The proof that \(K\cap dJd^{-1}=J_{\kappa/\mu}\) is similar.
\qed
\end{proof}

\begin{definition}\label{moving}
Given a tableau \(t\in\STD(\lambda/\mu)\), for each \(j\in[1,\lambda_1]\) let
\(M(t,j)\) be set of all integers~\(k\) such that
\(k\leqslant\min(\lambda_j^*-\lambda_{j+1}^*,\mu_j^*-\mu_{j+1}^*)\) and
\(t(i+k,j)<t(i,j+1)\) for all \(i\) such that \(\mu_j^*-k<i\leqslant\lambda_{j+1}^*\).
Define \(\moveup(t,j)=\max(M(t,j))\).
\end{definition}

\begin{remark}\label{movingremark}
Intuitively, \(\moveup(t,j)\)
is the maximal amount by which the \(j\)-th column of \(t\)
can be slid up while keeping the tableau standard. Observe that when
\(k\in  M(t,j)\), or, more generally, when
\(k\leqslant\min(\lambda_j^*-\lambda_{j+1}^*,\mu_j^*-\mu_{j+1}^*)\), the condition
\(\mu_j^*-k<i\leqslant\lambda_{j+1}^*\) is satisfied if and only if
\((i,j+1)\) and \((i+k,j)\) are both in \([\lambda/\mu]\). It is clear that 
\(0\in M(t,j)\) in all cases. Indeed, it is clear that \(M(t,j)=[0,\moveup(t,j)]\).
Note also that if \(\mu_j^*\geqslant\lambda_{j+1}^*\) then
\(\mu_j^*-\lambda_{j+1}^*\in M(t,j)\).
\end{remark}
Let \(\lambda/\mu\vdash n\) and \(t\in\STD(\lambda/\mu)\), and for each \(j\in[1,\lambda_1]\)
define \(\delta_j(t)=\sum_{l=j}^{\lambda_1}\moveup(t,l)\). Then
\begin{align*}
(\lambda_j^*-\delta_j(t))-(\lambda_{j+1}^*-\delta_{j+1}(t))
&=(\lambda_j^*-\lambda_{j+1}^*)-\moveup(t,j)\geqslant0\\
\noalign{\vskip-2 pt\hbox{and}\vskip-2 pt}
(\mu_j^*-\delta_j(t))-(\mu_{j+1}^*-\delta_{j+1}(t))
&=(\mu_j^*-\mu_{j+1}^*)-\moveup(t,j)\geqslant0,
\end{align*}
and it follows that
\(\{\,(i-\delta_j(t),j)\mid (i,j)\in[\lambda/\mu]\,\}\) is the diagram of a
skew partition \(\zeta/\eta\vdash n\) such that \(\zeta_j^*=\lambda_j^*-\delta_j(t)\)
and \(\eta_j^*=\mu_j^*-\delta_j(t)\) for all \(j\in[1,\lambda_1]\). Furthermore,
if \(t'\in\Tab_{\zeta/\eta}\) is defined by
\[
t'(i,j)=t(i+\delta_j(t),j)\qquad\text{for all \((i,j)\in[\zeta/\eta]\)}
\]
then \(t'\in\STD(\zeta/\eta)\), since if \(i\) and \(j\) are such that
\((i,j)\) and \((i,j+1)\) are both in \([\zeta/\eta]\) then
\[
t'(i,j)=t(h+\moveup(t,j),j)<t(h,j+1)=t'(i,j+1)
\]
where \(h=i+\delta_{j+1}(t)\). It is also clear that \(\delta_j(t')=0\) for all~\(j\).

\begin{definition}\label{squash}
Let \(\lambda/\mu\vdash n\) and \(t\in\STD(\lambda/\mu)\). We write \(\squash(t)\)
for the tableau \(t'\) defined in the above preamble, and we say that
\(t\) is \textit{squashed\/} if \(\squash(t)=t\).
\end{definition}

\begin{remark}\label{squashuniqueness}
Let \(\lambda/\mu\vdash n\) and \(\zeta/\eta\vdash n\), and suppose
that \(\lambda_j^*-\mu_j^*=\zeta_j^*-\eta_j^*\) for all~\(j\). Note that this
condition implies that \(J_{\lambda/\mu}=J_{\zeta/\eta}\). It is easily shown that
if \(t\in\STD(\lambda/\mu)\) and \(u\in\STD(\zeta/\eta)\) then \(\squash(t)=\squash(u)\)
if and only if \(\perm(t)=\perm(u)\), or, equivalently, the tableau \(u\) can be obtained
from the tableau \(t\) by sliding columns vertically. Hence, in the terminology
introduced at the start of this section, if \(J\subseteq\{s_1,s_2,\ldots,s_n\}\)
and \(w\in D_J\) then  \(\tab(J,w)=\squash(w\tab_{\lambda/\mu})\)
for all  \(\lambda/\mu\in\Skew(J,w)\).
\end{remark}

\begin{lemma}\label{squashtop}
Let \(\lambda/\mu\vdash n\). Then \(\squash(\tab^{\lambda/\mu})=\tab^{\zeta/\eta}\),
where \([\zeta/\eta]\) is obtained from \([\lambda/\mu]\) by removing all empty rows.
\end{lemma}

\begin{proof}
Let \(j\in[1,\lambda_1]\). Since all entries in any row of \(\tab^{\lambda/\mu}\) exceed all
entries in all earlier rows, it can be seen that
\(\moveup(\tab^{\lambda/\mu},j)=\max(0,\mu_j^*-\lambda_{j+1}^*)\), from which the
result follows.
\qed
\end{proof}

\begin{corollary}\label{squashtop-cor}
Let \(\lambda/\mu\vdash n\), and let \(r=\squash(\lessthan{\tab^{\lambda/\mu}}n)\)
and \(u=\squash(\morethan{\tab^{\lambda/\mu}}1)\). Let \(q=\lambda_1^*\) and
suppose that all rows of \([\lambda/\mu]\) are nonempty.
\setitemindent{ (iii)}
\begin{itemize}[topsep=1 pt]
\item[\textup{(i)}] If \(\mu_1<\lambda_1-1\) then
\(u=1+\tab^{\lambda/\nu}\),
where \(\nu_1=\mu_1+1\) and \(\nu_i=\mu_i\) for all~\(i>1\).
\item[\textup{(ii)}] If \(\mu_1=\lambda_1-1\) then
\(u=1+\tab^{\zeta/\nu}\),
where \(\nu_i=\mu_{i+1}\) and \(\zeta_i=\lambda_{i+1}\) for all~\(i\geqslant1\).
\item[\textup{(iii)}] If \(\mu_q<\lambda_q-1\) then
\(r=\tab^{\kappa/\mu}\), where 
\(\kappa_q=\lambda_q-1\), and \(\kappa_i=\lambda_i\) for all~\(i\ne q\).
\item[\textup{(iv)}] If \(\mu_q=\lambda_q-1\) then
\(r=\tab^{\kappa/\xi}\), where \(\kappa_i=\lambda_i\) and \(\xi_i=\mu_i\)
for all~\(j<q\), and \(\kappa_i=\xi_i=0\) for all \(i\geqslant q\).\qed
\end{itemize}
\end{corollary}

\begin{definition}\label{WGtableau}
Let \(W=W_n\), where \(n\) is a positive integer. A \textit{\(W\!\)-graph determining
tableau\/} is a tableau 
\(t\in\STD(\lambda/\mu)\), where \(\lambda/\mu\vdash n\), such that
\((\mathscr{I}(\perm(t)),J_{\lambda/\mu})\) is a \(W\!\)-graph ideal.
\end{definition}

\begin{remark}\label{squashWGT}
If \(t\in\STD(\lambda/\mu)\) is a \(W\!\)-graph determining tableau and
\(\squash(t)=t'\in\STD(\zeta/\eta)\), then \(t'\) is also a \(W\)-graph determining
tableau, since \(\perm(t')=\perm(t)\) and \(J_{\zeta/\eta}=J_{\lambda/\mu}\).
Moreover, if \(w\in W\) and \((\mathscr{I}(w),J)\) is a \(W\!\)-graph ideal,
then \(\tab(J,w)\) is a \(W\!\)-graph determining tableau. So classifying
\(W\!\)-graph determining tableaux is essentially the same as classifying
\(W\!\)-graph determining elements in type~\(A\).
\end{remark}
It follows from Remark \ref{D_Jcells} that if \(t\in\STD(\lambda/\mu)\)
is a \(W\!\)-graph determining tableau and \(J=J_{\lambda/\mu}\) then
\(\mathscr{I}(\perm(t))\) is a union of cells of \((D_J,J)\).
This observation motivates the following definition.

\begin{definition}\label{cellidealtableaux}Let \(W=W_n\), where \(n\) is
a positive integer, and let \(t\in\STD(\lambda/\mu)\), where
\(\lambda/\mu\vdash n\). We say that \(t\) is a \textit{cell ideal
generating tableau\/} for \(W\) if \(\mathscr{I}(\perm(t))\)
is a union of cells of \((D_J,J)\), where \(J=J_{\lambda/\mu}\).
\end{definition}

\begin{remark}\label{squashCIGT}Observe that if \(t\in\STD(\lambda/\mu)\) and
\(u\in\STD(\zeta/\eta)\) are such that \(\squash(t)=\squash(u)\), then \(t\) is
a cell ideal generating tableau if and only if \(u\) is a cell ideal generating
tableau, since \(\perm(t)=\perm(u)\) and \(J_{\lambda/\mu}=J_{\zeta/\eta}\).
\end{remark}

\begin{remark}\label{summary}
By Theorem~\ref{finalThr} and Lemma~\ref{squashtop}, every squashed maximal
tableau is a squashed \(W\!\)-graph determining tableau, and hence also
a squashed cell ideal generating tableau. Our main theorem below will
show that every squashed cell ideal generating tableau is maximal, so that
these three classes of tableaux concide.

Recall that if \(J\subseteq S_n\) and \(w\in D_J\) then \(t=\tab(J,w)\) is the
unique squashed standard tableau such that \(w=\perm(t)\) and \(J=J_{\lambda/\mu}\),
where \(\lambda/\mu\) is the shape of~\(t\). Thus classifying the squashed cell
ideal generating tableaux is the same as classifying the pairs \((w,J)\) such
that \(\mathscr{I}(w)\) is a union of cells of \((D_J,J)\), which in turn is the
same as classifying the pairs \((w,J)\) such that 
\(\{\,x\in W\mid w_J\leqslant\lside x\leqslant\lside ww_J\,\}\) is a nonempty
union of Kazhdan--Lusztig left cells.
\end{remark}

\begin{lemma}\label{restrictedcigt}
Let \(n>1\) be a positive integer and suppose that \(t\) is a cell ideal generating
tableau for \(W_n\). Then \(\lessthan tn\) and \(-1+(\morethan t1)\) are both
cell ideal generating tableaux for~\(W_{n-1}\).
\end{lemma}

\begin{proof}
Let \(\lambda/\mu\) be the shape of~\(t\), and let \(\kappa\) be defined by
\([\kappa]=[\lambda]\setminus\{t^{-1}(n)\}\), so that \(\kappa/\mu\) is the
shape of \(\lessthan tn\). Let \(K=S_n\setminus\{s_{n-1}\}\), so that
\(W_{n-1}\) can be identified with the standard parabolic subgroup \(W_K\)
of \(W=W_n\) and let \(J=J_{\lambda/\mu}\subseteq S_n\). Let \(w=\perm(t)\) and
\(v=\perm(\lessthan tn)\), so that \(v\) is the left \(W_K\)-component of~\(w\),
by Lemma~\ref{restriction}.

Let \(d\) be the \(D_K^{-1}\)-component of~\(w\), so that \(w=vd\) and \(d\in D_{K,J}\).
Since \(l(vd)=l(v)+l(d)\) it follows that if \(y\in W_K\) then \(y\leqslant\lside v\)
if and only if \(yd\leqslant\lside vd\). Hence
\(\mathscr{I}(w)d^{-1}\cap W_K=\mathscr{I}(v)\). But since \(t\) is a cell ideal
generating tableau it follows that \(\mathscr{I}(w)\) is a union of cells of
\((D_J,J)\), and hence, by Proposition~\ref{cellrestrict},
\(\mathscr{I}(w)d^{-1}\cap W_K\) is a union of cells of \((D^K_M,M)\), where
\(M=K\cap dJd^{-1}\). Since \(M=J_{\kappa/\mu}\), by Lemma~\ref{restriction},
this shows that \(\mathscr{I}(v)\) is a union of \((D_M,M)\) cells in~\(W_{n-1}\),
and hence \(\lessthan tn\) is a cell ideal generating tableau for~\(W_{n-1}\),
as claimed.

The proof of the other part is similar, using \(L=S_n\setminus\{s_1\}\) instead
of~\(K\), identifying \(W_L\) with \(W_{[2,n]}\), and applying the obvious
isomorphism \(W_{[2,n]}\cong W_{n-1}\).
\qed
\end{proof}

\begin{lemma}\label{exceptionals}
Suppose that \(n\geqslant 3\). The following two tableaux \(t\) and \(u\), namely
\[
t\ =\ \vcenter{\offinterlineskip\small
\hbox{\hskip 13.4pt\vrule height0.4 pt depth0 pt width 108.4pt}
\hbox{\hskip 13.4pt\vrule height 9 pt depth 4 pt\hbox to 13pt{\hfil\(2\)\hfil}\vrule
\hbox to 13pt{\hfil\(3\)\hfil}\vrule\hbox to 30pt{\hfil\(\cdots\)\hfil}\vrule
\hbox to 25pt{\hfil\(n-2\)\hfil}\vrule
\hbox to 25pt{\hfil\(n-1\)\hfil}\vrule} 
\hrule
\hbox{\vrule height 9 pt depth 4 pt\hbox to 13pt{\hfil\(1\)\hfil}\vrule
\hbox to 13pt{\hfil\(n\)\hfil}\vrule}
\hrule width 27.2 pt}\ ,
\qquad
u\ =\ \vcenter{\offinterlineskip\small
\hbox{\hskip 82.6pt\vrule height0.4 pt depth0 pt width 39.2pt}
\hbox{\hskip 82.6pt\vrule height 9 pt depth 4 pt\hbox to 25pt{\hfil\(1\)\hfil}\vrule
\hbox to 13pt{\hfil \(n\)\hfil}\vrule}
\hrule
\hbox{\vrule height 9 pt depth 4 pt\hbox to 13pt{\hfil\(2\)\hfil}\vrule
\hbox to 13pt{\hfil\(3\)\hfil}\vrule
\hbox to 30pt{\hfil\(\cdots\)\hfil}\vrule
\hbox to 25pt{\hfil\(n-2\)\hfil}\vrule
\hbox to 25pt{\hfil\(n-1\)\hfil}\vrule} 
\hrule width 108.4pt}\ ,
\]
are not cell ideal generating tableaux.
\end{lemma}
\begin{proof}
Let \(x=\perm(t)\) and \(J=J_{\lambda/\mu}\), where \(\lambda/\mu\) is the shape
of~\(t\). Then \(J=\{s_2\}\) and \(x\) is the \((n-2)\)-cycle
\((n, n-1, \ldots,4, 3)\). That is, \(x=s_{n-1}s_{n-2}\cdots s_3\). Clearly
\(s_1\nleqslant\lside x\), and hence \(s_1s_2\notin\mathscr{I}(x)w_J=\mathscr{I}(x){s_2}\).
Since obviously \(s_2\in\mathscr{I}(x)w_J\), we see that
\(\mathscr{I}(x)w_J\cap W_{\{s_1,s_2\}}\) is not a union of cells of \(W_{\{s_1,s_2\}}\).
So it follows from Lemma~\ref{celllemma} that \(\mathscr{I}(x)w_J\) is not a union of
cells of~\(W\), and so \(\mathscr{I}(x)\) is not a union of cells of \((D_J,J)\),
by Remark \ref{D_Jcells}.

Similarly, let \(y=\perm(u)\) and \(K=J_{\kappa/\nu}\), where \(\kappa/\nu\) is
the shape of~\(u\). Then \(K=\{s_{n-2}\}\) and \(y=s_1s_2\cdots s_{n-3}\), and
\(\mathscr{I}(y)w_K\) contains \(s_{n-2}\) but not \(s_{n-1}s_{n-2}\).
So \(\mathscr{I}(y)w_K\cap W_{\{s_{n-2},s_{n-1}\}}\) is not a union of cells
of \(W_{\{s_{n-2},s_{n-1}\}}\), and \(\mathscr{I}(y)\) is not a union of cells
of \((D_K,K)\).
\qed
\end{proof}
We now come to the main theorem of this paper.

\begin{theorem}
If \(n\) is a positive integer and \(W\) is the symmetric group on
\(\{1,2,\ldots,n\}\), then a squashed skew tableau is a \(W\!\)-graph ideal
determining tableau if and only if it is a maximal skew tableau,
and if and only if it is a cell ideal generating tableau for~\(W\!\).
\end{theorem}

\begin{proof}
As explained in Remark~\ref{summary} above, it only remains to prove that every
cell ideal generating tableau is maximal. Suppose that this
is false, and let \(n\) be the minimal counterexample. Let
\(W=W_n\), and let \(t\) be a squashed cell ideal generating tableau that is not of
the form~\(\tab^{\lambda/\mu}\). Let \(\lambda/\mu\) be the shape of~\(t\), and let
\(J=J_{\lambda/\mu}\).

Note that modifying \(t\) by removing empty columns does not change \(J\) or
\(\perm(t)\), and so does not alter the fact that \(t\) is a cell ideal
generating tableau. Nor does it alter the fact that \(t\) is squashed and
\(t\ne\tab^{\lambda/\mu}\). So we shall assume that \(t\) has no empty columns.
Since \(t\) is squashed it has no empty rows.

If \(n=1\) there is only one skew diagram \([\lambda/\mu]\) with no empty rows or
columns, and only one standard \((\lambda/\mu)\)-tableau. Since this contradicts
\(t\ne \tab^{\lambda/\mu}\), it follows that \(n>1\).

If \(n=2\) then there are exactly four standard skew tableaux with no empty rows
or columns, namely
\[
\text{\small \begin{ytableau}
1&2
\end{ytableau}\ ,
\qquad
\begin{ytableau}
1\\2
\end{ytableau}\ ,
\qquad
\begin{ytableau}
\none&1\\2
\end{ytableau}
\qquad\text{\normalsize and}\qquad
\begin{ytableau}
\none&2\\1
\end{ytableau}\ ,}
\]
the last of which is not squashed, and hence not equal to~\(t\). So \(t\)
is one of the others, and in each case we see that \(t\) is maximal,
contrary to the choice of~\(t\). So \(n>2\).

Write \(t_n=\squash(\lessthan tn)\) and \(t_1=\squash(\morethan t1)\).
It follows from Lemma~\ref{restrictedcigt} and Remark~\ref{squashCIGT}
that \(t_n\) and \(-1+t_1\) are cell ideal generating tableaux, and so, by
the minimality of our counterexample, it follows that \(t_n=\tab^{\zeta/\eta}\)
and \(-1+t_1=\tab^{\theta/\xi}\) for some \(\zeta/\eta\vdash n-1\) and
\(\theta/\xi\vdash n-1\). Let \(g=\col(t,1)\) and \(h=\col(t,n)\).

Since \(\squash(\morethan{t_n}1)=\squash(\morethan{(\lessthan tn)}1)
=\squash(\lessthan{(\morethan t1)}n)=\squash(\lessthan{t_1}n)\), by
Remark~\ref{squashuniqueness}, we deduce that
\begin{equation}\label{key}
\squash(\morethan{\tab^{\zeta/\eta}}1)=1+\squash(\lessthan{\tab^{\theta/\xi}}n).
\end{equation}
Note that since \(t_n\) and \(t_1\) are both squashed, neither
\([\zeta/\eta]\) nor \([\theta/\xi]\) has empty rows.  Let \(q\) be the
number of rows of~\([\theta/\xi]\). Our strategy is to compare the shapes
of the tableaux on the left and right sides of Eq.~\eqref{key}, using
Corollary~\ref{squashtop-cor}. There are four cases.

\Case{1.}
Suppose that \(1\) is the unique entry in the first row of
\(t_n\) and \(n\) is not the unique entry in the last row of~\(t_1\).
Thus \(\eta_1=\zeta_1-1\) and \(\xi_q<\theta_q-1\). 
For the left hand side of Eq.~\eqref{key}, case~(ii) of
Corollary~\ref{squashtop-cor} applies; for the right hand side, case~(iii) of
Corollary~\ref{squashtop-cor} applies. We find that \(\zeta_i=\theta_{i-1}\)
for \(2\leqslant i\leqslant q\), while \(\zeta_{q+1}=\theta_q-1\).
Note also that \(\theta_q=\col(t_1,n)=\col(t,n)=h\).

Suppose first that \(q\geqslant 2\).
Then \(\zeta_q=\theta_{q-1}\geqslant\theta_q\), and so we can define
a partition \(\zeta'\) by setting \(\zeta'_i=\zeta_i\) for \(i\leqslant q\)
and \(\zeta'_{q+1}=\theta_q=\zeta_{q+1}+1\). Note that
\(\zeta'_i\geqslant\eta_i\) for all~\(i\), and that
\([\zeta'/\eta]=[\zeta/\eta]\cup\{(q+1,h)\}\). Writing
\(t'=\tab^{\smash{\zeta'/\eta}}\), the maximal tableau of shape
\(\zeta'/\eta\), we see that \(t_n=\tab^{\smash{\zeta/\eta}}\)
is the restriction of \(t'\) to \([\zeta/\eta]\), and hence
\(\col(t',a)=\col(t_n,a)=\col(t,a)\) for all \(a\in[1,n-1]\).
Furthermore, \(\col(t',n)=h=\col(t,n)\). So the columns of \(t'\)
are the same as the columns of~\(t\), and since \(t'\) and
\(t\) are both squashed it follows that \(t=t'\). So \(t\) is
maximal, contrary to hypothesis.

It remains to consider the possibility that \(q=1\), which means that
\[
t_1 = \vcenter{\offinterlineskip\small
\hrule
\hbox{\vrule height 9 pt depth 4 pt\hbox to 13pt{\hfil\(2\)\hfil}\vrule
\hbox to 13pt{\hfil\(3\)\hfil}\vrule
\hbox to 30pt{\hfil\(\cdots\)\hfil}\vrule
\hbox to 25pt{\hfil\(n-1\)\hfil}\vrule
\hbox to 13pt{\hfil\(n\)\hfil}\vrule}
\hrule}\,.
\]
It will follow from the reasoning below that \(t_1\) does not have an empty
first column (to the left of the~\(2\)), but in any case we can say that
the partition \(\theta\) has only one part, \(\theta_1=\xi_1+n-1\), and
\(\xi\) has at most one part. Moreover, \(\col(t,a)=\col(t_1,a)=\xi_1+a-1\)
for all \(a\in[2,n]\).

The partition \(\zeta\) has two parts, namely \(\zeta_1=g=\col(t,1)\)
and \(\zeta_2=\theta_1-1=\xi_1+n-2\). So \(g\geqslant \xi_1+n-2\). 
We can see now that \(\xi_1=0\), since \(t\) has no empty columns
and \(\col(t,a)>\xi_1\) for all \(a\in [1,n]\). It is also clear that
\(g\leqslant n\), since \(t\) has at most~\(n\) columns. So we have
\(\col(t,a)=\col(t_1,a)=a-1\) for all \(a\in[2,n]\), and
\(\col(t,1)\in\{n-2,\,n-1,\,n\}\). Note that since \(t\) is squashed,
it is uniquely determined by its columns.

If \(\col(t,1)=n\) then the columns (from left to right) are
\(\{2\}\), \(\{3\}\), \dots, \(\{n\}\)
and~\(\{1\}\), and so \(t=\tab^{\lambda/\mu}\) with
\(\lambda/\mu=(n,n-1)/(n-1)\), contradicting the fact that \(t\) is not
a maximal tableau. If \(\col(t,1)=n-1\) then the columns of \(t\) are
\(\{2\}\), \(\{3\}\), \dots, \(\{n-1\}\)  and \(\{1,n\}\), from which
it follows that \(t=\tab^{\lambda/\mu}\) with \(\lambda/\mu=(n-1,n-1)/(n-2)\),
which again contradicts the fact that \(t\) is not a maximal tableau. Finally,
if \(\col(t,1)=n-2\) then the columns of \(t\) are \(\{2\}\),~\(\{3\}\),
\dots, \(\{n-2\}\), \(\{1,n-1\}\) and \(\{n\}\), and it follows
that \(t\) is the tableau \(u\) of Lemma~\ref{exceptionals},
contradicting the assumption that \(t\) is a cell ideal generating tableau.

\Case{2.}
Suppose that \(1\) is not the unique entry in the first row of \(t_n\) and
\(n\) is the unique entry in the last row of~\(t_1\). Thus
\(\eta_1<\zeta_1-1\) and \(\xi_q=\theta_q-1\). This time case~(i) of
Corollary~\ref{squashtop-cor} applies to the left hand side of
Eq.~\eqref{key}, and case~(iv) applies to the right hand side. The
partition \(\zeta\) has \(q-1\) parts, \(\zeta_i=\theta_i\) and
\(\eta_i=\xi_i\) for \(2\leqslant i\leqslant q-1\), while \(\zeta_1=\theta_1\)
and \(g=\eta_1+1=\xi_1\).

Suppose first that \(q-1\geqslant 2\). Then \(\eta_1\geqslant \eta_2=\xi_2\), and
hence we can define a partition \(\xi'\) by putting \(\xi'_i=\xi_i\) for
\(i\geqslant 2\) and \(\xi'_1=\eta_1=\xi_1-1\). Note that
\(\xi'_i\leqslant\theta_i\) for all~\(i\), and that
\([\theta/\xi']=[\theta/\xi]\cup\{(1,g)\}\). Defining
\(t'\) to be the maximal tableau of shape
\(\theta/\xi'\), we see that \(1+t_1=1+\tab^{\smash{\theta/\xi}}\)
is the restriction of \(t'\) to \([\theta/\xi]\), and hence
\(\col(t',a)=\col(t_1,a)=\col(t,a)\) for all \(a\in[2,n]\).
Furthermore, \(\col(t',1)=g=\col(t,1)\). So the columns of \(t'\)
are the same as the columns of~\(t\), and since \(t'\) and
\(t\) are both squashed it follows that \(t=t'\). So \(t\) is
maximal, contrary to hypothesis.

It remains to consider the possibility that \(q-1=1\), so that
\(t_n=\tab^{\zeta/\eta}\) has only one row. Thus \(\zeta_1=g+n-2\)
and \(\eta_1=g-1\), and \(\col(t,a)=\col(t_n,a)=g+a-1\) for all \(a\in[1,n-1]\).
Note that \(g\leqslant 2\), since \(t\) has at most
\(n\) columns. If \(g=2\) then \(t\) has \(n\) columns, all with exactly
one entry, and it follows that the columns of \(t\) must be \(\{n\}\), \(\{1\}\),
\(\{2\}\), \dots, \(\{n-1\}\). Hence \(t=\tab^{\lambda/\mu}\) with
\(\lambda/\mu=(n,1)/(1)\), contradicting the fact that \(t\) is not maximal.
If \(g=1\) then
\(h=\col(t,n)\) must be \(1\)~or~\(2\), since \(n\) is in the first
nonempty column of~\(t_1\). If \(h=2\) then the columns of~\(t\)
are \(\{1\}\), \(\{2,n\}\), \(\{3\}\), \dots, \(\{n-1\}\), which implies
that \(t\) is the \(t\) of Lemma~\ref{exceptionals}, contradicting to the
fact that \(t\) is a cell ideal generating tableau. If \(h=1\)
then the columns of~\(t\) are \(\{1,n\}\), \(\{2\}\), \(\{3\}\), \dots,
\(\{n-1\}\), from which it follows that \(t=\tab^{\lambda/\mu}\) where
\(\lambda=(n,1)\) and \(\mu\) is empty. Again this contradicts the fact that
\(t\) is not maximal.

\Case{3.}
Suppose that \(1\) is the unique entry in the first row of \(t_n\)
and \(n\) is the unique entry in the last row of~\(t_1\). This time
case~(ii) of Corollary~\ref{squashtop-cor} applies to the left hand side of
Eq.~\eqref{key}, and case~(iv) applies to the right hand side. Both
\(t_1\) and \(t_n\) have~\(q\) rows, and the tableau in Eq.~\eqref{key}
has \(q-1\) rows.

We have \(\zeta_i=\theta_{i-1}\) and \(\eta_i=\xi_{i-1}\) for \(2\leqslant i\leqslant q\),
and it follows that \(\zeta_q=\theta_{q-1}\geqslant \theta_q\) and
\(\eta_q=\xi_{q-1}\geqslant\xi_q\). So there is a \(q+1\) row skew diagram
\([\zeta'/\eta']\) with \(\zeta_{q+1}'=\theta_q\) and \(\eta_{q+1}'=\xi_q\),
and \(\zeta_i'=\zeta_i\) and \(\eta_i'=\eta_i\) for \(i\leqslant q\).
Clearly \([\zeta'/\eta']=[\zeta/\eta]\cup\{(q+1,h)\}\). Defining
\(t'\) to be the maximal tableau of shape \(\zeta'/\eta'\), we see that
\(t_n=\tab^{\zeta/\eta}\) is the restriction of \(t'\) to \([\zeta/\eta]\).
Thus \(\col(t',a)=\col(t_n,a)=\col(t,a)\) for all \(a\in[1,n-1]\),
and \(\col(t',n)=h=\col(t,n)\). So \(t=t'\), contradicting the fact that
\(t\) is not maximal.

\Case{4.}
The only remaining possibility is that \(1\) is not the unique entry in the
first row of \(t_n\) and \(n\) is not the unique entry in the last row of~\(t_1\).
In this case \(t_n\) and \(t_1\) both have \(q\) rows, and the tableau in
Eq.~\eqref{key} also has \(q\) rows. Case~(i) of Corollary~\ref{squashtop-cor}
applies to the left hand side of Eq.~\eqref{key}, and case~(iii) applies to the
right hand side. We find that \(\theta_q=\zeta_q+1\) and \(\theta_i=\zeta_i\)
for \(i< q\), while \(\eta_1=\xi_1-1\) and \(\eta_i=\xi_i\) for \(i>1\).
Thus \(\eta_i\leqslant\xi_i\leqslant\zeta_i\leqslant\theta_i\) for all~\(i\),
and \(\theta/\eta\) is a skew partition. Furthermore,
\([\theta/\eta]=[\zeta/\eta]\cup\{(q,h)\}\). As in the previous cases it follows
that \(\col(t,a)=\col(\tab^{\theta/\eta},a)\) for all \(a\in [1,n]\), so that
\(t=\tab^{\theta/\eta}\), contradicting the fact that \(t\) is not maximal.
This final contradiction completes the proof.
\qed
\end{proof}

To conclude we state the following Corollary, whose proof was explained in
Remark~\ref{summary}.

\begin{corollary}
Let \(n\) be a positive integer and \((W,S)=(W_n,S_n)\), a Coxeter group
of type~\(A_{n-1}\). If \(w\in W\) and \(J\subseteq S\) then
\(\{\,x\in W\mid w_J\leqslant\lside x\leqslant\lside ww_J\,\}\) is a nonempty
union of Kazhdan--Lusztig left cells if and only if \(w=\perm(\tab^{\lambda/\mu})\)
and \(J=J_{\lambda/\mu}\) for some \(\lambda/\mu\vdash n\), and then \(x\in W\)
satisfies \(w_J\leqslant\lside x\leqslant\lside ww_J\) if and only if
\(xw_J\tab_{\lambda/\mu}\in\STD(\lambda/\mu)\).
\end{corollary}

\begin{acknowledgement}
I am grateful to A/Prof Robert B. Howlett for providing many improvements to the exposition
of this paper.
\end{acknowledgement}

\end{document}